\documentclass{amsart}
\usepackage{amssymb,mathtools}
\usepackage[british]{babel}
\usepackage{enumitem}
\usepackage[latin1]{inputenc}
\usepackage{url}
\usepackage{hyperref}

\newtheorem{theorem}{Theorem}
\numberwithin{theorem}{section}
\newtheorem{corollary}[theorem]{Corollary}
\newtheorem{lemma}[theorem]{Lemma}
\newtheorem{proposition}[theorem]{Proposition}
\theoremstyle{definition}
\newtheorem{definition}[theorem]{Definition}
\newtheorem{remark}[theorem]{Remark}

\DeclarePairedDelimiter{\ceil}{\lceil}{\rceil}
\newcommand\norm[1]{\left\lVert#1\right\rVert}

\title[Bounds for a nonlinear ergodic theorem for Banach spaces]{Bounds for a nonlinear ergodic theorem\\ for Banach spaces}
\author{Anton Freund and Ulrich Kohlenbach}

\address{Department of Mathematics, Technical University of Darmstadt, Schloss\-garten\-str.~7, 64289~Darmstadt, Germany}
\email{\{freund,kohlenbach\}@mathematik.tu-darmstadt.de}

\begin{document}

\begin{abstract}
We extract quantitative information (specifically, a rate of metastability in the sense of Terence Tao) from a proof due to Kazuo Kobayasi and Isao Miyadera, which shows strong convergence for {C}es\`aro means of non\-expansive maps on {B}anach spaces.
\end{abstract}


\maketitle

\section{Introduction}

Throughout this paper, we assume that $X$ is a uniformly convex real Banach space, and that $C\subseteq X$ is non-empty, closed and convex. Furthermore, we assume that the map $T:C\to C$ has a fixed point and is nonexpansive (i.\,e., that we have $\norm{Tx-Ty}\leq\norm{x-y}$ for all~$x,y\in C$). By
\begin{equation*}
S_nx:=\frac1{n}\sum_{i=0}^{n-1}T^ix
\end{equation*}
we denote the {C}es\`aro means with respect to~$T$. Our aim is a quantitative version of the following result due to K.~Kobayasi and I.~Miyadera:

\begin{theorem}[{\cite{KobMiy}}]\label{thm:KobMiy}
Consider $T:C\to C$  as above. Given $x\in C$, assume that the sequences $(\norm{T^nx-T^{n+i}x})_n$ converge uniformly in~$i$. We then have
\begin{equation*}
\lim_{n\to\infty}\norm{y-S_nT^kx}=0\quad\text{uniformly in~$k$},
\end{equation*}
for some fixed point~$y$ of~$T$.
\end{theorem}

Let us discuss what a quantitative version of the conclusion should look like. The proof by Kobayasi and Miyadera shows that the fixed point~$y$ from the theorem is the limit of the sequence~$(S_mT^mx)$. Hence we obtain
\begin{align}
\lim_{m,n\to\infty}\norm{S_mT^mx-S_nT^kx}&=0\quad\text{uniformly in~$k$},\label{eq:SnTk-univ-conv}\\
\lim_{n\to\infty}\norm{T^lS_nT^nx-S_nT^nx}&=0\quad\text{uniformly in~$l$}.\label{eq:SnTn-FP}
\end{align}
To see that we get (\ref{eq:SnTn-FP}), note that $Ty=y$ and the fact that~$T$ is nonexpansive yield
\begin{equation*}
\norm{T^lS_nT^nx-S_nT^nx}\leq\norm{T^lS_nT^nx-T^ly}+\norm{y-S_nT^nx}\leq 2\cdot\norm{y-S_nT^nx}.
\end{equation*}
Indeed, the conjunction of (\ref{eq:SnTk-univ-conv}) and (\ref{eq:SnTn-FP}) is equivalent to the conclusion of Theorem~\ref{thm:KobMiy} together with the result that~$(S_nT^nx)$ converges to a fixed point of~$T$. With respect to~(\ref{eq:SnTn-FP}), we note that uniformity in~$l$ is somewhat trivial in the presence of an actual fixed point $y$. It is less trivial when only approximate fixed points are available, and this will play a role in our quantitative analysis (cf.~the formulation of~(\ref{eq:SnTn-FP-metastab}) below). Let us also point out that statements (\ref{eq:SnTk-univ-conv}) and~(\ref{eq:SnTn-FP}) entail
\begin{equation}\label{eq:S_nx-asymptotic-reg}
\lim_{n\to\infty}\norm{TS_nT^kx-S_nT^kx}=0\quad\text{uniformly in~$k$}.
\end{equation}
This asymptotic regularity result is due to R.~Bruck~\cite{Bruck79} (in the Banach space case).

What, then, should a quantitative version of (\ref{eq:SnTk-univ-conv}) assert? As a first idea, we might look for a rate of convergence, i.\,e., for a function $\varepsilon\mapsto N(\varepsilon)$ with
\begin{equation*}
\norm{S_mT^mx-S_nT^kx}<\varepsilon\quad\text{for all $m,n\geq N=N(\varepsilon)$ and all $k\in\mathbb N$}.
\end{equation*}
In particular, $\varepsilon\mapsto N(\varepsilon/2)$ would be a Cauchy rate for the sequence~$(S_nx)$. It is known that such a rate cannot be computed (or given by a ``simple closed expression") in general (see~\cite[Theorem~5.1]{Avigad-Gerhardy-Towsner}). However, we will be able to construct a rate of metastability, i.\,e., a map $(\varepsilon,g,h)\mapsto\Phi(\varepsilon,g,h)$ that takes an $\varepsilon>0$ and functions $g,h:\mathbb N\to\mathbb N$ as input and ensures that we have
\begin{multline}\label{eq:SnTk-univ-metastab}
\norm{S_mT^mx-S_nT^kx}<\varepsilon\quad\text{for some $N\leq\Phi(\varepsilon,g,h)$ and}\\ \text{all $m,n\in[N,N+g(N)]$ and $k\leq h(N)$}.
\end{multline}
Note that this is still as strong as~(\ref{eq:SnTk-univ-conv}) from above: for if the latter fails, there is an $\varepsilon>0$ such that any $N$ admits $m,n\geq N$ and $k\in\mathbb N$ with $\norm{S_mT^mx-S_nT^kx}\geq\varepsilon$. If we set $g(N):=\max\{m,n\}-N$ and $h(N):=k$ for such numbers, (\ref{eq:SnTk-univ-metastab}) must fail. Similarly, our quantitative analysis of~(\ref{eq:SnTn-FP}) will yield a map $(\varepsilon,g,h)\mapsto\Psi(\varepsilon,g,h)$ with
\begin{multline}\label{eq:SnTn-FP-metastab}
\norm{T^lS_nT^nx-S_nT^nx}<\varepsilon\quad\text{for some $N\leq\Psi(\varepsilon,g,h)$ and}\\ \text{all $n\in[N,N+g(N)]$ and $l\leq h(N)$}.
\end{multline}
Following the structure of Kobayasi and Miyadera's proof, we will first construct~$\Psi$ and use it to define~$\Phi$. However, it will then turn out that one can switch to~$\Psi:=\Phi$, and that all desired properties can be satisfied simultaneously: there is a number $N\leq\Phi(\varepsilon,g,h)$ such that all $m,n\in[N,N+g(N)]$ and $k,l\leq h(N)$ validate both $\norm{S_mT^mx-S_nT^kx}<\varepsilon$ and $\norm{T^lS_nT^nx-S_nT^nx}<\varepsilon/2$. From these two bounds we get $\norm{T^lS_nT^k-S_nT^kx}<5\varepsilon/2$, which provides quantitative information about the asymptotic regularity result~(\ref{eq:S_nx-asymptotic-reg}) due to Bruck (see Corollary~\ref{cor:simultaneous-1to3} for all this). We have mentioned that one cannot expect a computable rate of convergence (rather than metastability) for~(\ref{eq:SnTk-univ-conv}). It is not clear whether rates of convergence are available in the case of (\ref{eq:SnTn-FP}) or~(\ref{eq:S_nx-asymptotic-reg}). The proof by Kobayasi and Miyadera~\cite{KobMiy} does not seem to yield such rates, while Bruck's~\cite{Bruck79} proof of~(\ref{eq:S_nx-asymptotic-reg}) remains to be analyzed. In the case of Hilbert space, a rate of convergence for~(\ref{eq:S_nx-asymptotic-reg}) is known (see~\cite[Lemma~3.4]{kohlenbach_baillon}).

The term ``metastability" for statements such as (\ref{eq:SnTk-univ-metastab}) and (\ref{eq:SnTn-FP-metastab}) has been coined by T.~Tao~\cite{tao-blog-book}. Even before Tao had introduced this terminology, the notion had been studied in mathematical logic, in particular in the proof mining program (see the textbook~\cite{kohlenbach-book}), with foundational work reaching back to K.~G\"odel. One interesting aspect of metastability is its connection with the number of $\varepsilon$-fluctuations (see the results by J.~Avigad and J.~Rute~\cite[Section~5]{Avigad-Rute}). Both experience and general metatheorems from logic (cf.~the end of this section) show that rates of metastability can be extracted from a wide range of mathematical proofs. Specifically, the present paper complements quantitative work on nonlinear non\-expansive operators that satisfy a 
condition due to Wittmann \cite{wittmann90} (cf.~Section~\ref{sect:asymp-isom-Hilbert}), notably by P.~Safarik~\cite{safarik12} (convergence of Ces\`aro means in Hilbert spaces) and the second author~\cite{kohlenbach-odd-operators} (convergence of iterates of asymptotically regular maps in Banach spaces). More generally, there is a wealth of quantitative results on the convergence of various iteration schemes in Hilbert, Banach and more general spaces (see e.\,g.\ \cite{Avigad-Gerhardy-Towsner,Kohlenbach-Leustean-09} for the linear and \cite{kohlenbach-mann-iterates,kohlenbach_browder-wittmann,kohlenbach_baillon,kohlenbach-PPA-2021,Kohlenbach-Leustean-EMS,kohlenbach-sipos-sunny,powell_banach19} for the nonlinear case).

As explained above, our quantitative analysis of Theorem~\ref{thm:KobMiy} (due to Kobayasi and Miyadera~\cite{KobMiy}) consists in the construction of maps $\Phi$ and $\Psi$ as in (\ref{eq:SnTk-univ-metastab}) and~(\ref{eq:SnTn-FP-metastab}). In the following, we specify the quantitative data that we consider as given. First, we assume that we have a bound~$b>0$ with
\begin{equation}\label{eq:bound-b-FP}
C\subseteq B_{b/2}:=\left\{x\in X\,\left|\,\norm{x}\leq\frac{b}{2}\right.\right\}.
\end{equation}
Note that such a bound exists if, and only if, $C$ has bounded diameter (specifically (\ref{eq:bound-b-FP}) yields $\norm{x-y}\leq b$ for $x,y\in C$). The version of Theorem~\ref{thm:KobMiy} for bounded~$C$ is not actually weaker: To obtain the full theorem, recall the assumption that~$T$ has a fixed point~$f$. For $x\in C$ and $r:=\norm{x-f}$, the set $D:=\{x\in C\,|\,\norm{x-f}\leq r\}\cap C$ is closed and convex with $T(D)\subseteq D\supseteq\{x,f\}$. In view of~$D\subseteq B_{r+\norm{f}}$, we can conclude by the seemingly weaker version of Theorem~\ref{thm:KobMiy}. In the context of Lemma~\ref{lem:co-p_approx} we will need a bound as in~(\ref{eq:bound-b-FP}), not just a bound on the diameter.

Secondly, to witness our standing assumption on the Banach space~$X$, we assume as given a modulus $\eta:(0,2]\to (0,1]$ of uniform convexity, so that we have
\begin{equation}\label{eq:modulus-unif-conv}
\norm{\frac{x+y}{2}}\leq 1-\eta(\varepsilon)\quad\text{for $\norm{x},\norm{y}\leq 1$ with $\norm{x-y}\geq\varepsilon$}.
\end{equation}
Sometimes it is convenient to have $\eta$ defined on $[0,\infty)$ with values 
in $[0,\infty)$ as a continuous function satisfying
\begin{equation}\label{eq:eta-convex}
0<s<t\quad\Rightarrow\quad 0=\eta(0)<\eta(s)<\eta(t)\text{ and }\frac{\eta(s)}{s}\leq\frac{\eta(t)}{t}.
\end{equation}
This form is used in Section~\ref{sect:bound_nonlinearity}, which exhibits
the quantitative content of an intermediate result by Bruck~\cite{Bruck81} and 
where it is shown 
how an arbitrary modulus as in (\ref{eq:modulus-unif-conv}) can be converted 
into a new modulus $\eta'$ with the additional properties. Conversely, given 
$\eta':[0,\infty)\to[0,\infty)$ one can simply take $\eta:(0,2]\to (0,1]$ with 
$\eta(\varepsilon):=\min\{ 1,\eta'(\varepsilon)\}$ (note in particular that (\ref{eq:modulus-unif-conv}) is trivial when $\varepsilon>2$ and false or void when $\eta(\varepsilon)>1$).  The property (\ref{eq:eta-convex}) can equivalently be expressed as 
$\eta(s)=s\cdot\tilde\eta(s)$ for a non-decreasing and continuous $\tilde\eta:[0,\infty)\to[0,\infty)$ with $\tilde\eta(s)>0$ for~$s>0$. Let us also recall that the $L^p$-spaces admit the natural modulus 
\begin{equation*}
\eta(\varepsilon)=\begin{cases}
\frac{p-1}{8}\varepsilon^2 & \text{if $1<p<2$,}\\
\frac{\varepsilon^p}{p\cdot 2^p} & \text{if $2\le p<\infty$,}
\end{cases}
\end{equation*}
for which (\ref{eq:modulus-unif-conv}) and~(\ref{eq:eta-convex}) are satisfied (see~\cite{hanner56} and also~\cite[Section~3]{Kohlenbach_Krasnoselski-Ishikawa}).

Thirdly, given that $X$ is a uniformly convex Banach space, so is $X^2$ with the norm defined by $\norm{(x,y)}_2=(\norm{x}^2+\norm{y}^2)^{1/2}$. Indeed, assume (\ref{eq:modulus-unif-conv}) holds for $(X,\norm{\cdot})$ and a modulus $\eta$ that satisfies $\eta(\varepsilon)\leq\varepsilon/4$ and $\eta(\varepsilon)=\varepsilon\cdot\tilde\eta(\varepsilon)$ with non-decreasing~$\tilde\eta$ (which 
e.g.\ follows if $\eta$ satisfies (\ref{eq:eta-convex})). An analysis of the proof in~\cite{Clarkson-convex-spaces} shows that (\ref{eq:modulus-unif-conv}) remains valid when $(X,\norm{\cdot})$ and $\eta$ are replaced by the space $(X^2,\norm{\cdot}_2)$ and the modulus $\eta_2:(0,2]\to(0,1]$ given by
\begin{equation*}
\eta_2(\varepsilon):=\delta\left(\frac{\varepsilon}{8\sqrt{2}}\cdot\eta\left(\frac{\varepsilon}{\sqrt{2}}\right)\right)\quad\text{with}\quad\delta(\varepsilon):=\frac{\varepsilon^2}{8}.
\end{equation*}
Note that $\eta_2$ satisfies (\ref{eq:eta-convex}), if $\eta$ does. Now the uniformly convex space $(X^2,\norm{\cdot}_2)$ is, in particular, $B$-convex. By a characterization due to G.~Pisier~\cite{pisier73} (cf.~the proof of~\cite[Theorem~1.1]{Bruck81}), this means that there are~$c>0$ and $q>1$ with the following property: for all independent random variables $Z_1,\ldots,Z_n$ with values in~$X^2$, the expected values satisfy
\begin{equation}\label{eq:B-convex-expectations}
\mathbb E\left(\norm{\sum_{i=1}^n Z_i}_2^q\right)\leq c^q\cdot\sum_{i=1}^n\mathbb E\left(\norm{Z_i}_2^q\right).
\end{equation}
For our quantitative analysis, we assume that we are given such $c$ and~$q$. This additional data could be avoided, i.\,e., expressed in terms of our modulus~$\eta$ of uniform convexity (as guaranteed by the metatheorem cited below). 
Indeed one can explicitly construct suitable $c,q$ in terms of 
$\eta_2$ via an analysis of the proof in~\cite{pisier73}. In fact, since 
that proof only uses that $X$ is uniformly nonsquare in the sense of 
James, it suffices to use one nontrivial value of $\eta_2$, e.g.~$\eta_2(1)$. 
In the case of $L^p$-spaces one can take $q:=\min\{ 2,p\},$ where the 
optimal constant $c$ has been computed in \cite{Haagerup} (see 
\cite[Section~9.2]{ledoux-talagrand}). We include nevertheless 
$c,q$ among our input data, because this simplifies matters and is computationally harmless: Note that $c$ and~$q$ depend on the space only. Also, the complexity class of our bounds does not depend on~$c$ and~$q$, because the latter are numbers rather than functions.

Finally, for given~$x\in C$ we abbreviate
\begin{equation*}
\alpha^i_n:=\norm{T^nx-T^{n+i}x}.
\end{equation*}
Since~$T$ is nonexpansive, we always have $\alpha^i_n\geq\alpha^i_{n+1}\geq 0$, so that each of the sequences~$(\alpha^i_n)$ converges. A central assumption of Theorem~\ref{thm:KobMiy} demands that the rate of convergence (but not necessarily the limit) is independent of the number~$i$. In terminology due to Bruck~\cite{Bruck78}, this means that~$T$ is asymptotically isometric on the set~$\{x\}$. As a quantitative version of this assumption, we suppose that we are given a rate of metastability, i.\,e., a map $(\varepsilon,g,h)\mapsto A(\varepsilon,g,h)$ that guarantees
\begin{multline}\label{eq:metastab-alpha}
\left|\alpha^i_m-\alpha^i_n\right|<\varepsilon\quad\text{for some $N\leq A(\varepsilon,g,h)$}\\ \text{and all $m,n\in[N,N+g(N)]$ and $i\leq h(N)$}.
\end{multline}
In order to apply our quantitative result, one will have to provide a map~$A$ with this property. We mention three situations where this is possible: First, assume that $T$ satisfies Wittmann's \cite{wittmann90} condition $\norm{Tx+Ty}\leq\norm{x+y}$ (which is e.\,g.\ the case when $C=-C$ and $T$ is odd in addition 
to being nonexpansive) and is asymptotically regular with given rate (\mbox{$\norm{T^{n+1}x-T^nx}\to 0$} for $n\to\infty$, which holds e.\,g.\ for averaged maps). As shown by the second author (see~\cite{kohlenbach-odd-operators} and the generalization in~\cite[Section~3]{kohlenbach-banach-geodesic-2016}), one can then construct a rate of metastability that witnesses $\norm{T^ix-T^jx}\to 0$ for $i,j\to\infty$. Such a rate is readily transformed into a map~$A$ that validates~(\ref{eq:metastab-alpha}). Secondly, the assumption that $T$ is asymptotically regular can be dropped in the Hilbert space case. Finally, one can satisfy~(\ref{eq:metastab-alpha}) when $(T^nx)$ has a convergent subsequence (even in Banach space). A quantitative analysis of the second and third situation (which are mentioned by Kobayasi and Miyadera~\cite{KobMiy}) is given in Section~\ref{sect:asymp-isom-Hilbert} of the present paper. Let us point out that all three constructions of~$A$ yield a rate of metastability rather than convergence. In this respect, it is also interesting to consider the beginning of Section~\ref{sect:fixed-points}, where a ``limsup$\,\leq\,$liminf"-argument from the proof of~\cite[Lemma~2]{KobMiy} forces us to settle for metastability.

Overall, our aim is to construct maps $\Phi$ and $\Psi$ as in (\ref{eq:SnTk-univ-metastab}) and~(\ref{eq:SnTn-FP-metastab}). These will only depend on given maps~$A,\eta$ and numbers~$b,c,q$ as in~(\ref{eq:bound-b-FP}-\ref{eq:metastab-alpha}). In addition to this quantitative data, we keep the assumption that $T:C\to C$ has a fixed point and is nonexpansive on the convex subset~$C$ of our uniformly convex Banach space. Concerning complexity, it will be straightforward to observe that all our constructions are primitive recursive in the sense of S.~Kleene (see e.\,g.~\cite[Section~3.4]{kohlenbach-book}). Let us recall that Safarik~\cite{safarik12} has previously obtained primitive recursive bounds in the case of Hilbert space.

To conclude the introduction, we return to the topic of logical metatheorems. In order to determine the precise bounds $\Phi$ and $\Psi$, it is of course necessary to consider the proof of Theorem~\ref{thm:KobMiy} in detail. However, the fact that one can extract suitable $\Phi$ and $\Psi$ is guaranteed in advance, by the second author's general result~\cite[Theorem~3.30]{kohlenbach-FA} on uniformly convex normed linear spaces. We only sketch why the latter applies (cf.~\cite[Section~3]{safarik12} for a detailed discussion in a related case): The cited metatheorem covers, roughly speaking, results of the form ``for~all--exists". A convergence statement such as (\ref{eq:SnTk-univ-conv}) does not have this form, as the existential claim (``there is an~$N$") is followed by a universal quantification (``for all~$m,n\geq N$"). On the other hand, the metastable version ``for all $\varepsilon,g,h$ there is a $\Phi(\varepsilon,g,h)$ as in~(\ref{eq:SnTk-univ-metastab})" does have the required form; here it is crucial that the quantification over $k,m,n$ and~$N$ inside~(\ref{eq:SnTk-univ-metastab}) ``does not count", because it only refers to numbers below a given bound. Similarly, the assumption associated with (\ref{eq:metastab-alpha}) contains essentially no existential quantification when we treat~$A$ as a given function (which we may assume to be a majorant, namely of the function $A^-$ providing the least $N$ satisfying~(\ref{eq:metastab-alpha})). In this situation, the cited metatheorem predicts that there are computable maps $\Phi$ and $\Psi$ that only depend on our bound~$b$ with $C\subseteq B_{b/2}$, on the given function~$A$, and on the modulus~$\eta$ of uniform convexity. Furthermore, the proof of the metatheorem suggests a general strategy for the extraction of $\Phi$ and~$\Psi$. Hence the logical background is useful in practice and interesting as a uniform explanation. At the same time, each concrete application can be presented without any reference to logic, as the following sections testify.
\begin{remark}[For logicians]
Officially, the aforementioned metatheorem requires $A$ to be a {\bf 
strong} majorant for some $A^-$ satisfying ~(\ref{eq:metastab-alpha}). 
However, 
strong majorization is only needed when dealing with proofs 
whose quantitative analysis requires so-called bar recursion,
which is not the case here, while otherwise ordinary majorization can be 
used in the monotone functional interpretation proving the metatheorem 
(see \cite{kohlenbach-book}, Remark 17.37). 
\end{remark}

\section{Nonlinearity and convex combinations}\label{sect:bound_nonlinearity}

In this section, we discuss quantitative aspects of a result due to Bruck~\cite{Bruck81}. Specifically, we construct an increasing function~$\gamma:[0,\infty)\to[0,\infty)$ such that
\begin{equation}\label{eq:bound-nonlinearity}
\gamma\left(\norm{T\left(\sum_{i=1}^n\lambda_ix_i\right)-\sum_{i=1}^n\lambda_iTx_i}\right)\leq\max_{1\leq i,j\leq n}(\norm{x_i-x_j}-\norm{Tx_i-Tx_j})
\end{equation}
holds for any convex combination $\sum\lambda_ix_i$ (i.\,e., we require $\lambda_i\geq 0$ and $\sum\lambda_i=1$). We note that most quantitative information in this section is already quite explicit in Bruck's original presentation. Nevertheless, it will be important to streamline some constructions for our purpose (cf.~the paragraph after Definition~\ref{def:gamma_n}).

Bruck first constructs functions~$\gamma=\gamma_n$ that satisfy~(\ref{eq:bound-nonlinearity}) for fixed~$n$. In a second step, he achieves independence of~$n$ by diagonalizing over these functions. A more common way to assert~(\ref{eq:bound-nonlinearity}) for $n=2$ is to say that~$T$ is of type~($\gamma$). This reveals that the case~$n=2$ coincides with~\cite[Lemma~1.1]{Bruck79}. The following makes the computational information explicit. By a standing assumption from the introduction, the function $\eta:[0,\infty)\to[0,\infty)$ is a modulus of uniform convexity for~$X$, while~$b>0$ bounds the diameter of~$C\subseteq X$, the domain of our map~$T:C\to C$.

\begin{definition}\label{def:gamma_2}
Let $\gamma_2:[0,\infty)\to[0,\infty)$ be given by $\gamma_2(t)=\min\left\{t,\frac{b}{2}\cdot\eta\left(\frac{4t}{b}\right)\right\}$.
\end{definition}

The next lemma shows that~(\ref{eq:bound-nonlinearity}) holds for $\gamma=\gamma_2$ and fixed~$n=2$. The additional properties ensure that we have a strictly increasing inverse~$\gamma_2^{-1}:[0,\infty)\to[0,\infty)$ (with $\gamma_2^{-1}(t)\geq t$ due to the minimum above), as used in the proof of Lemma~\ref{lem:Bruck2.1}. 

\begin{lemma}\label{lem:Bruck1.1}
The function~$\gamma_2$ is strictly increasing, unbounded and continuous with minimal value $\gamma_2(0)=0$. For all $x_1,x_2\in C$ and $\lambda\in[0,1]$ we have
\begin{equation*}
\gamma_2(\norm{T(\lambda x_1+(1-\lambda)x_2)-(\lambda Tx_1+(1-\lambda)Tx_2)})\leq\norm{x_1-x_2}-\norm{Tx_1-Tx_2}.
\end{equation*}
\end{lemma}
\begin{proof}
The first sentence of the lemma is immediate by the corresponding properties of~$\eta$, which hold by a standing assumption from the introduction (note that~(\ref{eq:eta-convex}) yields $\eta(t)\geq\eta(1)\cdot t$ for $t\geq 1$). For the remaining claim, we follow the proof of~\cite[Lemma~1.1]{Bruck79}. As in the latter, the value of
\begin{equation}\label{eq:Lem1.1-proof}
2\lambda(1-\lambda)\cdot\norm{x_1-x_2}\cdot\eta\left(\frac{\norm{\lambda Tx_1+(1-\lambda)Tx_2-T(\lambda x_1+(1-\lambda)x_2)}}{\lambda(1-\lambda)\cdot\norm{x_1-x_2}}\right)
\end{equation}
is smaller than or equal to~$\norm{x_1-x_2}-\norm{Tx_1-Tx_2}$ (unless $\lambda\in\{0,1\}$ or $x_1=x_2$ and the claim is trivial). Now recall the standing assumption that~$\eta$ is convex. More specifically, from~(\ref{eq:eta-convex}) we readily get $t\cdot\eta(r/t)\leq s\cdot\eta(r/s)$ for $r\geq 0$ and $0<s\leq t$. With $s=\lambda(1-\lambda)\cdot\norm{x_1-x_2}\leq b/4=t$, we see that~(\ref{eq:Lem1.1-proof}) is larger than or equal to
\begin{equation*}
\frac{b}{2}\cdot\eta\left(\frac{4\cdot\norm{\lambda Tx_1+(1-\lambda)Tx_2-T(\lambda x_1+(1-\lambda)x_2)}}{b}\right).
\end{equation*}
Hence the definition of~$\gamma_2$ (even without the minimum) is as required.
\end{proof}

We have reproduced part of the proof from~\cite{Bruck79} in order to show how the convexity of~$\eta$ is used. As promised in the introduction, we now recall how a convex modulus can be constructed.

\begin{remark}
Assume that~$\eta_0:(0,2]\to(0,1]$ is any modulus of uniform convexity for our Banach space~$X$, which means that~(\ref{eq:modulus-unif-conv}) holds with~$\eta_0$ at the place of~$\eta$. Define~$\eta_1:(0,2]\to(0,1]$ by setting
\begin{equation*}
\eta_1(\varepsilon)=\sup\{\eta_0(\varepsilon')\,|\,\varepsilon'\in(0,\varepsilon]\}.
\end{equation*}
Then~(\ref{eq:modulus-unif-conv}) does still hold with~$\eta_1$ at the place of~$\eta$ (since $\eta_1(\varepsilon)>1-\norm{(x-y)/2}$ entails $\eta_0(\varepsilon')>1-\norm{(x-y)/2}$ for some~$\varepsilon'\leq\varepsilon$). The point is that $\eta_1$ is increasing (not necessarily strictly). We also write $\eta_1:[0,\infty)\to[0,1]$ for the extension of this function by the values $\eta_1(0)=0$ and $\eta_1(\varepsilon)=\eta_1(2)$ for~$\varepsilon>2$. Then $\eta_1$ is still increasing and hence Riemann integrable. As in the proof of~\cite[Lemma~1.1]{Bruck79}, we now define~$\eta:[0,\infty)\to[0,\infty)$ by
\begin{equation*}
\eta(\varepsilon)=\frac12\cdot\int_0^{\varepsilon}\eta_1(t)\,\mathrm{d}t.
\end{equation*}
For $\varepsilon\in(0,2]$ we have $\eta(\varepsilon)\leq\varepsilon\cdot\eta_1(\varepsilon)/2\leq\eta_1(\varepsilon)$, so that (\ref{eq:modulus-unif-conv}) holds for~$\eta$. Note that this extends to all~$\varepsilon\in[0,\infty)$, by $\eta(0)=0$ and the trivial reason mentioned in the introduction. Furthermore, $\eta$ is strictly increasing and continuous. For $0<s\leq t$ we can split the integral to get $2\cdot\eta(t)\geq 2\cdot\eta(s)+(t-s)\cdot\eta_1(s)$ and then
\begin{equation*}
2\cdot(s\cdot\eta(t)-t\cdot\eta(s))\geq (t-s)\cdot(s\cdot\eta_1(s)-2\cdot\eta(s))\geq 0.
\end{equation*}
We thus have $\eta(s)/s\leq\eta(t)/t$, as required by the convexity property from~(\ref{eq:eta-convex}). In conclusion, $\eta$ satisfies all standing assumptions from the introduction. To estimate the cost of these assumptions, we observe $\eta_1(\varepsilon)\leq 2\cdot\eta(2\varepsilon)/\varepsilon$ for $\varepsilon\in(0,2]$. As noted in the introduction, the given construction is not required in the case of $L^p$-spaces with $1<p<\infty$,
where the specific $\eta$ given does already satisfy 
(\ref{eq:modulus-unif-conv}) and (\ref{eq:eta-convex}).
\end{remark}

We will see that the following functions validate (\ref{eq:bound-nonlinearity}) for arbitrary but fixed~$n$.

\begin{definition}\label{def:gamma_n}
We construct functions $\gamma_n:[0,\infty)\to[0,\infty)$ by recursion on~$n\geq 2$. The base case is provided by Definition~\ref{def:gamma_2}, while the step is given by
\begin{equation*}
\gamma_{n+1}(t)=\min\left\{\gamma_n(t),\gamma_2\left(\frac{\gamma_n(t/2)}{3}\right)\right\}.
\end{equation*}
\end{definition}

Let us point out that we do not make the functions~$\gamma_n$ convex, as this property will not be needed beyond the proof of Lemma~\ref{lem:Bruck1.1}. This allows us to give a somewhat simpler definition than Bruck. One other point is important: The proof of~\cite[Lemma~2.1]{Bruck81} suggests a recursive construction of~$\gamma_n^{-1}$ rather than~$\gamma_n$. We will consider inverses in the verification below, but we have avoided them in the construction itself, because this gives more control on the complexity of our bounds. Indeed, a function and its inverse do not generally belong to the same complexity class. As an example for logicians, we mention that the function~$F_{\varepsilon_0}^{-1}$ in~\cite{freund-proof-length} is primitive recursive while~$F_{\varepsilon_0}$ is not.

\begin{lemma}\label{lem:Bruck2.1}
The functions~$\gamma_n$ are strictly increasing, unbounded and continuous with $\gamma_n(0)=0$. For each fixed~$n\geq 2$, inequality~(\ref{eq:bound-nonlinearity}) is valid for~$\gamma=\gamma_n$.
\end{lemma}
\begin{proof}
A straightforward induction over~$n$ yields the first sentence of the lemma. The point is that we can now consider the inverses~$\gamma_n^{-1}:[0,\infty)\to[0,\infty)$, which are strictly increasing as well. By the proof of~\cite[Lemma~2.1]{Bruck81}, the claim that~(\ref{eq:bound-nonlinearity}) holds for any convex combination (in the domain of~$T:C\to C$) reduces to
\begin{equation*}
\gamma_{n+1}^{-1}(s)\geq\gamma_2^{-1}(s)+\gamma_n^{-1}(s+2\cdot\gamma_2^{-1}(s)).
\end{equation*}
Given that Definition~\ref{def:gamma_2} forces~$\gamma_2(r)\leq r$, we inductively get $\gamma_{n+1}(r)\leq\gamma_n(r)\leq r$ and hence $r\leq\gamma_n^{-1}(r)$. This ensures that we have
\begin{equation*}
\gamma_2^{-1}(s)+\gamma_n^{-1}(s+2\cdot\gamma_2^{-1}(s))\leq 2\cdot\gamma_n^{-1}(3\cdot\gamma_2^{-1}(s))=:t.
\end{equation*}
We can thus conclude
\begin{equation*}
\gamma_{n+1}(\gamma_2^{-1}(s)+\gamma_n^{-1}(s+2\cdot\gamma_2^{-1}(s)))\leq\gamma_{n+1}(t)\leq\gamma_2\left(\frac{\gamma_n(t/2)}{3}\right)=s.
\end{equation*}
Since~$\gamma_{n+1}^{-1}$ is increasing, this yields the open claim.
\end{proof}

In order to obtain~(\ref{eq:bound-nonlinearity}) for a function~$\gamma$ that is independent of~$n$, Bruck argues that any convex combination can be approximated by one with a bounded number of summands. More specifically, this observation is applied to convex combinations of elements $(x,Tx)\in C\times C\subseteq X^2$. Here $X^2$ is a uniformly convex Banach space with norm given by $\norm{(x,y)}_2=(\norm{x}^2+\norm{y}^2)^{1/2}$, as discussed in the introduction. For a number~$p\geq 1$ and a subset $M\subseteq X^2$, we put
\begin{equation*}
\operatorname{co}_p(M)=\left\{\left.\sum_{i=1}^n\lambda_iz_i\,\right|\,z_i\in M\text{ and }\lambda_i\geq 0\text{ with }\sum_{i=1}^n\lambda_i=1\text{ for }n\leq p\right\}.
\end{equation*}
Let us note that we can always arrange~$n=p$ by repeating some of the~$z_i$. Also write $\operatorname{co}(M)$ for the full convex hull (i.\,e., the union over all $\operatorname{co}_p(M)$ with~$p\geq 1$), and recall that $B_r=\{z\in X^2\,|\,\norm{z}_2\leq r\}$ is the closed ball with radius~$r$. By a standing assumption, we have $c>0$ and $q>1$ that validate~(\ref{eq:B-convex-expectations}) from the introduction. The proof of~\cite[Theorem~1.1]{Bruck81} contains the following computational information.

\begin{lemma}\label{lem:co-p_approx}
We have $\operatorname{co}(M)\subseteq\operatorname{co}_p(M)+B_{r\cdot\varepsilon}$ when $M\subseteq B_r$ and $2cp^{(1-q)/q}\leq\varepsilon$.
\end{lemma}
\begin{proof}
In the proof of~\cite[Theorem~1.1]{Bruck81}, Bruck seems to claim that the result holds for~$2cp^{1/q}<\varepsilon$. This appears counterintuitive, since $p\mapsto 2cp^{1/q}$ is strictly increasing while $\operatorname{co}_p(M)$ grows with~$p$. We recall Bruck's proof to show that it actually yields our bound (cf.~the 
proof of Theorem~6.2 in~\cite{Dilworth00} for a similar reasoning). By a straightforward rescaling we may assume~$r=1$. Consider any convex combination
\begin{equation*}
z=\sum_{i=1}^n\lambda_iz_i\in\operatorname{co}(M)\subseteq B_1.
\end{equation*}
For~$p$ as in the lemma, consider independent $X^2$-valued random variables~$Z_1,\ldots,Z_p$ with identical distribution given by
\begin{equation*}
Z_1,\ldots,Z_p\,\stackrel{\text{iid}}{\sim}\,\frac{Z-z}{p}\quad\text{for}\quad Z=z_i\text{ with probability }\lambda_i.
\end{equation*}
Now $(2/p)^q$ bounds all possible values of~$\norm{Z_i}_2^q$, and hence its expectation. By~(\ref{eq:B-convex-expectations}) we can conclude
\begin{equation*}
\mathbb E\left(\norm{\sum_{j=1}^pZ_j}_2^q\right)\leq(2c)^q\cdot p^{1-q}.
\end{equation*}
In particular, the right side must bound at least one possible value of~$\norm{\sum Z_j}_2^q$. More explicitly, there must be some event~$\omega:\{1,\ldots,p\}\to\{1,\ldots,n\}$ (under which $Z_j$ assumes value $(z_{\omega(j)}-z)/p$) such that we have
\begin{equation*}
\norm{z-\sum_{j=1}^p\frac{1}{p}\cdot z_{\omega(j)}}_2=\norm{\sum_{j=1}^p\frac{z_{\omega(j)}-z}{p}}_2\leq 2cp^{\frac{1-q}{q}}\leq\varepsilon.
\end{equation*}
This shows that our given $z\in\operatorname{co}(M)$ lies in $\operatorname{co}_p(M)+B_{\varepsilon}$, as desired.
\end{proof}

Following the construction by Bruck, we now diagonalize over the functions~$\gamma_n$, in order to achieve independence of~$n$.
 
\begin{definition}\label{def:gamma}
Let $\gamma:[0,\infty)\to[0,\infty)$ be given by $\gamma(0)=0$ and, for $t>0$,
\begin{equation*}
\gamma(t)=\gamma_{p(t)}\left(\frac{t}{3}\right)\quad\text{with}\quad p(t)=\max\left\{2,\ceil*{\left(\frac{6bc}{\sqrt{2}t}\right)^{q/(q-1)}}\right\}.
\end{equation*}
\end{definition}

The next proof is very close to the one of~\cite[Theorem~2.1]{Bruck81}. We provide details in order to show how the previous constructions come together.

\begin{proposition}\label{prop:Bruck2.1}
The function~$\gamma$ is strictly increasing and validates~(\ref{eq:bound-nonlinearity}) for any convex combination of elements $x_1,\ldots,x_n\in C$ with arbitrary~$n\geq 1$.
\end{proposition}
\begin{proof}
The function $t\mapsto p(t)$ is decreasing, as $q>1$ holds by a standing assumption from the introduction. Since Definition~\ref{def:gamma_n} ensures $\gamma_{n+1}(t)\leq\gamma_n(t)$, this yields
\begin{equation*}
s<t\quad\Rightarrow\quad \gamma(s)=\gamma_{p(s)}\left(\frac{s}{3}\right)\leq\gamma_{p(t)}\left(\frac{s}{3}\right)<\gamma_{p(t)}\left(\frac{t}{3}\right)=\gamma(t).
\end{equation*}
To establish~(\ref{eq:bound-nonlinearity}), consider a convex combination $\sum\lambda_ix_i$ of $x_1,\ldots,x_n\in C$. We set
\begin{equation*}
t:=\norm{T\left(\sum_{i=1}^n\lambda_ix_i\right)-\sum_{i=1}^n\lambda_iTx_i}
\end{equation*}
and assume~$t>0$, as the remaining case is trivial. The crucial idea (from the proof by Bruck~\cite{Bruck81}) is to apply Lemma~\ref{lem:co-p_approx} to $M=\{z_1,\ldots,z_n\}$ with $z_i=(x_i,Tx_i)$. By standing assumption~(\ref{eq:bound-b-FP}) we have $C\subseteq B_{b/2}$ in~$X$ and hence $M\subseteq C\times C\subseteq B_{b/\sqrt{2}}$ in~$X^2$ (recall $\norm{(x,y)}_2=(\norm{x}^2+\norm{y}^2)^{1/2}$). One readily checks
\begin{equation*}
\varepsilon:=\frac{\sqrt{2}t}{3b}\geq 2c\cdot p(t)^{(1-q)/q}.
\end{equation*}
For this~$\varepsilon$ and with~$p=p(t)$, Lemma~\ref{lem:co-p_approx} yields a subset $\{z_{i_1},\ldots,z_{i_p}\}\subseteq M$ and coefficients $\mu_1,\ldots,\mu_p\geq 0$ with $\sum\mu_j=1$ such that
\begin{equation*}
\norm{\sum_{i=1}^n\lambda_iz_i-\sum_{j=1}^p\mu_j z_{i_j}}_2\leq\frac{t}{3}.
\end{equation*}
In view of~$\norm{x},\norm{y}\leq\norm{(x,y)}_2$ and since $T$ is nonexpansive, this yields
\begin{equation*}
\norm{T\left(\sum_{i=1}^n\lambda_ix_i\right)-T\left(\sum_{j=1}^p\mu_j x_{i_j}\right)}\leq\frac{t}{3}\quad\text{and}\quad\norm{\sum_{i=1}^n\lambda_iTx_i-\sum_{j=1}^p\mu_j Tx_{i_j}}\leq\frac{t}{3}.
\end{equation*}
Using the triangle inequality, we can conclude
\begin{equation*}
t\leq\frac{t}{3}+\norm{T\left(\sum_{j=1}^p\mu_jx_{i_j}\right)-\sum_{j=1}^p\mu_jTx_{i_j}}+\frac{t}{3}.
\end{equation*}
Hence the remaining norm on the right must be at least~$t/3$. By Lemma~\ref{lem:Bruck2.1} we get
\begin{equation*}
\gamma(t)=\gamma_p\left(\frac{t}{3}\right)\leq\max_{1\leq k,l\leq p}(\norm{x_{i_k}-x_{i_l}}-\norm{Tx_{i_k}-Tx_{i_l}}).
\end{equation*}
This suffices to conclude, as the maximum on the right becomes larger when we admit all $i,j\in\{1,\ldots,n\}$, rather than those of the form $i=i_k$ and $j=i_l$ only.
\end{proof}

\section{Nonlinearity and Ces\`aro means}

Recall that $S_nx$ denotes the Ces\`aro mean $(T^0x+\ldots+T^{n-1}x)/n$. If~$T$ is nonlinear, then it may fail to commute with~$S_n$. This failure can be measured by
\begin{equation*}
\beta^l_n:=\norm{S_nT^{l+n}x-T^lS_nT^nx}.
\end{equation*}
More generally, we can recover these values as $\beta_n^l=\beta_{n,n}^l$ for
\begin{equation*}
\beta_{m,n}^l:=\norm{\frac{S_mT^{l+m}x+S_nT^{l+n}x}{2}-T^l\left(\frac{S_mT^mx+S_nT^nx}{2}\right)}.
\end{equation*}
Kobayasi and Miyadera~\cite[Lemma~1]{KobMiy} show that, under the assumptions made in our Theorem~\ref{thm:KobMiy} (which is their Theorem~1), we have
\begin{equation*}
\lim_{n\to\infty}\beta_n^l=\lim_{m,n\to\infty}\beta_{m,n}^l=0.
\end{equation*}
In the present section, we establish a rate of metastability for this result. Let us recall that we have only assumed a rate of metastability rather than convergence in~(\ref{eq:metastab-alpha}). If we had assumed a rate of convergence there, then we would get one for the~$\beta^l_{m,n}$ here (simply read $g(N)=\infty=h(N)$ below). However, we would then be forced to downgrade from convergence to metastability later, as explained in the introduction. In view of this fact, it is reasonable to work with a rate of metastability from the outset, as this makes for a computationally weaker assumption. Concerning the next definition, recall that our standing assumptions provide a map $(\varepsilon,g,h)\mapsto A(\varepsilon,g,h)$ that validates~(\ref{eq:metastab-alpha}). The function~$\gamma$ is given by Definition~\ref{def:gamma}.

\begin{definition}\label{def:metastab-beta}
For $\varepsilon>0$ and $g,h:\mathbb N\to\mathbb N$ we put $B(\varepsilon,g,h):=A(\gamma(\varepsilon),g',h')$ with $g'(N):=N+2\cdot g(N)+h(N)$ and $h'(N):=2\cdot(N+g(N))$.
\end{definition}

By analyzing the proof of~\cite[Lemma~1]{KobMiy}, we see the following:

\begin{lemma}\label{lem:KobMiy-Lem1}
For any $\varepsilon>0$ and $g,h:\mathbb N\to\mathbb N$, there is an~$N\leq B(\varepsilon,g,h)$ with
\begin{equation*}
\beta_{m,n}^l\leq\varepsilon\quad\text{for all }m,n\in[N,N+g(N)]\text{ and all }l\leq h(N).
\end{equation*}
In particular we get $\beta_n^l\leq\varepsilon$ under the same conditions on~$n$ and~$l$.
\end{lemma}
To be precise, we point out that $\beta_{m,n}^l$ is only defined for~$m,n>0$. In order to avoid this condition, one can simply declare $\beta_{0,n}^l=0=\beta_{m,0}^l$.
\begin{proof}
We first observe that Proposition~\ref{prop:Bruck2.1} remains valid with~$T^l:C\to C$ at the place of~$T:C\to C$, with the same~$\gamma:[0,\infty)\to[0,\infty)$ for all~$l\in\mathbb N$. Indeed, a glance at Definitions~\ref{def:gamma_2},~\ref{def:gamma_n} and~\ref{def:gamma} reveals that the relevant quantitative information depends on our Banach space~$X$ and the bounded subset~$C\subseteq X$ only (cf.~the standing assumptions (\ref{eq:bound-b-FP}-\ref{eq:metastab-alpha}) from the introduction). The qualitative assumption that~$T$ is nonexpansive holds of~$T^l$ as well. As in the proof of~\cite[Lemma~1]{KobMiy}, we now apply~(\ref{eq:bound-nonlinearity}) with~$T^l$ and $m+n$ at the place of~$T$ and~$n$, respectively, and with
\begin{equation*}
\begin{cases}
\begin{alignedat}{5}
x_i&{}=T^{m+i-1}x&\quad\text{and}\quad &\lambda_i&&{}=1/(2m)\quad && \text{for}\quad 1\leq i\leq m,\\
x_i&{}=T^{n+i-(m+1)}x&\quad\text{and}\quad &\lambda_i&&{}=1/(2n)\quad && \text{for}\quad m+1\leq i\leq m+n.
\end{alignedat}
\end{cases}
\end{equation*}
One readily checks that this yields
\begin{equation*}
\sum_{i=1}^{m+n}\lambda_ix_i=\frac{S_mT^mx+S_nT^nx}{2}\quad\text{and}\quad\sum_{i=1}^{m+n}\lambda_iT^lx_i=\frac{S_mT^{l+m}x+S_nT^{l+n}x}{2}.
\end{equation*}
Hence inequality~(\ref{eq:bound-nonlinearity}) amounts to
\begin{equation*}
\gamma(\beta_{m,n}^l)\leq\max\left\{\norm{x_i-x_j}-\norm{T^lx_i-T^lx_j}\,:\,1\leq i,j\leq m+n\right\}.
\end{equation*}
Writing $\alpha_n^i=\norm{T^nx-T^{n+i}x}$ as in the introduction, we get
\begin{equation}\label{eq:estimate-beta}
\gamma(\beta_{m,n}^l)\leq\max\{\alpha_k^p-\alpha_{l+k}^p\,|\,\min\{m,n\}\leq k<2\cdot\max\{m,n\}>p\}.
\end{equation}
Let $g'$ and $h'$ be given as in Definition~\ref{def:metastab-beta}. By a standing assumption from the introduction, pick an $N\leq A(\gamma(\varepsilon),g',h')$ such that (\ref{eq:metastab-alpha}) holds with $\gamma(\varepsilon),g'$ and $h'$ at the place of $\varepsilon,g$ and $h$, respectively. To establish the present lemma, consider arbitrary~$m,n\in[N,N+g(N)]$ and $l\leq h(N)$. For all~$k,l,p$ as in~(\ref{eq:estimate-beta}) we have
\begin{equation*}
N\leq k\leq k+l<2\cdot(N+g(N))+h(N)=N+g'(N)
\end{equation*}
and $p<h'(N)$. Hence the current version of~(\ref{eq:metastab-alpha}) yields $\alpha_k^p-\alpha_{l+k}^p<\gamma(\varepsilon)$ in all cases that are relevant for~(\ref{eq:estimate-beta}). The latter thus entails $\gamma(\beta_{m,n}^l)<\gamma(\varepsilon)$. Proposition~\ref{prop:Bruck2.1} tells us that~$\gamma$ is strictly increasing, so that we get $\beta_{m,n}^l<\varepsilon$ as desired.
\end{proof}

\section{Ces\`aro means and fixed points}\label{sect:fixed-points}

This section completes the quantitative analysis of Kobayasi and Miyadera's results in~\cite{KobMiy}, as sketched in the introduction.

We first analyze the preliminary result from Lemma~2 of~\cite{KobMiy}, which asserts that
\begin{equation*}
\theta_n^f:=\norm{S_nT^nx-f}
\end{equation*}
converges whenever~$f=Tf$ is a fixed point. From a methodological standpoint, this can be seen as the most interesting part of our quantitative analysis: it is here that we are forced to extract a rate of metastability rather than convergence. In order to make this transparent, we recall the original proof of~\cite[Lemma~2]{KobMiy}. For a fixed point $f=Tf$ and arbitrary $m\geq 1$ and $n\in\mathbb N$, the cited proof shows
\begin{equation}\label{eq:proof-KobMiy-Lem2}
\theta_{m+n}^f\leq\theta_m^f+\frac{m-1}{m+n}\cdot\norm{x-f}+\frac{1}{m+n}\cdot\sum_{i=0}^{m+n-1}\beta_m^{n+i},
\end{equation}
with $\beta^l_n=\norm{S_nT^{l+n}x-T^lS_nT^nx}$ as in the previous section. The proof then concludes with a ``limsup$\,\leq\,$liminf"-argument, which can be spelled out as follows: Write $\theta^f$ for the limit inferior of the sequence $(\theta^f_n)$. According to \cite[Lemma~1]{KobMiy} we have $\lim_{n\to\infty}\beta^l_n=0$ uniformly in~$l$ (cf.~our Lemma~\ref{lem:KobMiy-Lem1}). Given~$\varepsilon>0$, there must thus be a number $m\in\mathbb N$ such that we have
\begin{equation*}
\text{(i) $\theta^f_k\geq\theta^f-\varepsilon/4$ for~$k\geq m$,}\quad\text{(ii) $\beta^l_m\leq\varepsilon/4$ for~$l\in\mathbb N$,}\quad\text{(iii) $\theta^f_m\leq\theta^f+\varepsilon/4$.}
\end{equation*}
Put $N:=\max\{m,\ceil{4(m-1)\cdot\norm{x-f}/\varepsilon}\}$. For $j\geq N$ (think~$j=m+n$) we can combine~(\ref{eq:proof-KobMiy-Lem2}) and~(ii) to get $\theta^f_j\leq\theta^f_m+\varepsilon/2$, which by~(iii) yields $\theta^f_j\leq\theta^f+3\varepsilon/4$. Together with~(i) we get~$|\theta^f_j-\theta^f_k|\leq\varepsilon$ for $j,k\geq N$, as needed to show that $(\theta^f_n)$ is a Cauchy sequence. Our focus on the Cauchy property (rather than on convergence to the limit~$\theta^f$) assimilates the argument to the quantitative analysis below.

From a computational viewpoint, the previous argument appears problematic, because we do not know how fast the limit inferior is approximated. More explicitly, there is no obvious bound on a number~$m$ that would satisfy (i) or~(iii) above. Nevertheless, a modified argument does reveal quantitative information: If~(ii) holds for sufficiently many~$l$ (which can be ensured by Lemma~\ref{lem:KobMiy-Lem1}), then we obtain $\theta^f_j\leq\theta^f_m+\varepsilon/2$ as above. Without reference to~$\theta^f$, we can conclude that $|\theta^f_j-\theta^f_k|\leq\varepsilon$ holds unless we have~$\theta^f_k<\theta^f_m-\varepsilon/2$. In the latter case, we repeat the argument with~$k$ at the place of~$m$. Crucially, there can only be finitely many repetitions of this type, as~$\theta_n^f$ cannot become negative. For an explicit bound, use $f=Tf$ to get
\begin{equation}\label{eq:bound-theta^f_n}
\theta_n^f=\norm{f-\frac{1}{n}\cdot\sum_{i=0}^{n-1}T^ix}\leq\frac{1}{n}\cdot\sum_{i=0}^{n-1}\norm{T^ix-T^if}\leq\norm{x-f}\leq b.
\end{equation}
Here $b$ is a bound on the diameter of~$C\subseteq B_{b/2}$, the domain of our nonexpansive map $T:C\to C$ (cf.~standing assumption~(\ref{eq:bound-b-FP}) from the introduction). The preceding discussion is somewhat informal, but it helps to motivate the following:

\begin{definition}\label{def:Theta}
Consider a function~$g:\mathbb N\to\mathbb N$. As usual, we write $g^M$ for the monotone bound given by $g^M(n):=\max_{m\leq n}g(m)$. For~$\varepsilon>0$ we put
\begin{equation*}
g_\varepsilon(n):=n'+g^M(n')\quad\text{with}\quad n'=\max\left\{n,\ceil*{\frac{4nb}{\varepsilon}}\right\}.
\end{equation*}
Now consider the iterates~$g_\varepsilon^{(i)}:\mathbb N\to\mathbb N$ (with corrected start value) given by
\begin{equation*}
g_\varepsilon^{(0)}(n)=\max\{1,n\}\quad\text{and}\quad g_\varepsilon^{(i+1)}(n)=g_\varepsilon\left(g_\varepsilon^{(i)}(n)\right).
\end{equation*}
Finally, define a map $(\varepsilon,g)\mapsto\Theta(\varepsilon,g)$ by setting
\begin{equation*}
\Theta(\varepsilon,g):=g_\varepsilon^{(K+1)}\left(B\left(\frac{\varepsilon}{4},g_\varepsilon^{(K+1)},2\cdot g_\varepsilon^{(K+1)}\right)\right)\quad\text{with}\quad K=\ceil*{\frac{2b}{\varepsilon}}.
\end{equation*}
Here $B$ is the map from Definition~\ref{def:metastab-beta}.
\end{definition}

As a systematic explanation for logicians, we point out that $\Theta$ can be seen as the combination of two rates of metastability (cf.~Theorem~5.8 in the preprint version~\href{https://arxiv.org/abs/1412.5563}{arXiv:1412.5563} of~\cite{kohlenbach-leustean-nicolae_fejer}): It is well known that a non-increasing sequence in~$[0,b]$ admits a rate of metastability that depends on~$b$ only (see~\cite[Proposition~2.27]{kohlenbach-book}). We have a similar rate here, given that the $\theta^f_n$ are ``almost" non-increasing by~(\ref{eq:proof-KobMiy-Lem2}). This rate is combined with the rate $B$, which ensures that the last summand in~(\ref{eq:proof-KobMiy-Lem2}) is indeed small (see Lemma~\ref{lem:KobMiy-Lem1}). The fact that $\Theta$ combines two rates is made explicit in Corollary~\ref{cor:Theta^B} below (note that $\Theta^B$ is a minor variant of~$\Theta$). Let us now present our metastable version of~\cite[Lemma~2]{KobMiy}:

\begin{proposition}\label{prop:KobMiy-Lem2}
Let $f=Tf$ be a fixed point. For any~$\varepsilon>0$ and $g:\mathbb N\to\mathbb N$, there is an $N\leq\Theta(\varepsilon,g)$ such that $|\theta_m^f-\theta_n^f|<\varepsilon$ holds for all~$m,n\in [N,N+g(N)]$.
\end{proposition}
\begin{proof}
For~$K=\ceil{2b/\varepsilon}$, Lemma~\ref{lem:KobMiy-Lem1} yields an~$N_0\leq B(\varepsilon/4,g_\varepsilon^{(K+1)},2\cdot g_\varepsilon^{(K+1)})$ with
\begin{equation}\label{eq:beta-for-Theta}
\beta_m^l\leq\frac{\varepsilon}{4}\quad\text{for all }m\in\left[N_0,N_0+g_\varepsilon^{(K+1)}(N_0)\right]\text{ and }l\leq 2\cdot g_\varepsilon^{(K+1)}(N_0).
\end{equation}
We will show that some $N\in[N_0,g_\varepsilon^{(K+1)}(N_0)]$ validates the proposition. To see that we get~$N\leq\Theta(\varepsilon,g)$, it suffices to observe that each iterate~$g_\varepsilon^{(i)}$ is increasing. Inductively, this reduces to the same statement about~$g_\varepsilon$, which holds by construction. Aiming at a contradiction, we now assume: for any number $N\in[N_0,g_\varepsilon^{(K+1)}(N_0)]$ there are $m,n\in[N,N+g(N)]$ with $|\theta_m^f-\theta_n^f|\geq\varepsilon$. We shall construct a sequence of numbers $n(i)\leq g_\varepsilon^{(i)}(N_0)$ with $\theta^f_{n(i)}\leq b-i\varepsilon/2$ by recursion on $i\leq K+1$, which contradicts $\theta^f_{n(K+1)}\geq 0$. In the base case, set $n(0)=g_\varepsilon^{(0)}(N_0)$ and note $\theta^f_{n(0)}\leq b$ due to~(\ref{eq:bound-theta^f_n}). In the recursion step from $i\leq K$ to $i+1$, we consider
\begin{equation*}
N:=\max\left\{n(i),\ceil*{\frac{4b\cdot n(i)}{\varepsilon}}\right\}.
\end{equation*}
By~(\ref{eq:proof-KobMiy-Lem2}) for $m=n(i)$ and $n=N-n(i)+k$ (with arbitrary~$k$) we get
\begin{equation}\label{eq:proof-KobMiy-Lem2-for-Theta}
\theta^f_{N+k}\leq\theta^f_{n(i)}+\frac{\varepsilon}{4}+\frac{1}{N+k}\cdot\sum_{j=0}^{N+k-1}\beta_{n(i)}^{N-n(i)+k+j}.
\end{equation}
In order to apply~(\ref{eq:beta-for-Theta}), we verify that the indices of~$\beta$ lie in the appropriate interval of metastability. We anticipate that we will choose an $n(i+1)$ above $N\geq n(i)$, so that we may inductively assume $n(i)\geq n(0)\geq N_0$. As the function~$g_\varepsilon$ is increasing with $g_\varepsilon(j)\geq j$, we also have
\begin{equation*}
n(i)\leq N\leq N+g^M(N)=g_\varepsilon(n(i))\leq g_\varepsilon\left(g_\varepsilon^{(i)}(N_0)\right)\leq g_\varepsilon^{(K+1)}(N_0).
\end{equation*}
Furthermore, for $k\leq g(N)$ and $j<N+k$ we obtain
\begin{equation*}
N-n(i)+k+j<2\cdot(N+g^M(N))\leq 2\cdot g_\varepsilon^{(K+1)}(N_0).
\end{equation*}
In view of these bounds, we can combine (\ref{eq:beta-for-Theta}) and (\ref{eq:proof-KobMiy-Lem2-for-Theta}) to get
\begin{equation*}
\theta^f_{N+k}\leq\theta^f_{n(i)}+\frac{\varepsilon}{2}\quad\text{for}\quad k\leq g(N).
\end{equation*}
On the other hand, we have $|\theta^f_m-\theta^f_n|\geq\varepsilon$ for some $m,n\in [N,N+g(N)]$, by the contradictory assumption for the present $N\in[N_0,g_\varepsilon^{(K+1)}(N_0)]$. Thus there must be a~$k\leq g(N)$ with $\theta^f_{N+k}\leq\theta^f_{n(i)}-\varepsilon/2$. In order to complete the recursion step, we set $n(i+1):=N+k$ for some such $k$. Note that we indeed get
\begin{gather*}
n(i+1)\leq N+g^M(N)\leq g_\varepsilon^{i+1}(N_0),\\
\theta^f_{n(i+1)}\leq\theta^f_{n(i)}-\frac{\varepsilon}{2}\leq b-\frac{i\cdot\varepsilon}{2}-\frac{\varepsilon}{2}=b-\frac{(i+1)\cdot\varepsilon}{2},
\end{gather*}
so that $n(i+1)$ retains the properties that were promised above.
\end{proof}

We will want to bound $|\theta_m^f-\theta_n^f|$ and $\beta^l_m$ simultaneously, i.\,e., on the same interval of metastability. In the present case, it suffices to tweak our previous construction:

\begin{definition}
Extending Definition~\ref{def:Theta}, we set
\begin{equation*}
\Theta^B(\varepsilon,g,h):=g_\varepsilon^{(K+1)}\left(B\left(\frac{\varepsilon}{4},g_\varepsilon^{(K+2)},\max\left\{2\cdot g_\varepsilon^{(K+1)},h^M\circ g_\varepsilon^{(K+1)}\right\}\right)\right),
\end{equation*}
still with $K=\ceil{2b/\varepsilon}$ and $h^M(n)=\max_{m\leq n}h(m)$.
\end{definition}

As promised, this yields a simultaneous bound:

\begin{corollary}\label{cor:Theta^B}
For a fixed point $f=Tf$ and arbitrary $\varepsilon>0$ and $g,h:\mathbb N\to\mathbb N$, there is an $N\leq\Theta^B(\varepsilon,g,h)$ such that we have both $\beta_{m,n}^l\leq\varepsilon/4$ and $|\theta^f_m-\theta_n^f|<\varepsilon$ for all $m,n\in[N,N+g(N)]$ and $l\leq h(N)$.
\end{corollary}
\begin{proof}
From Lemma~\ref{lem:KobMiy-Lem1} we obtain a number
\begin{equation*}
N_0\leq B\left(\frac{\varepsilon}{4},g_\varepsilon^{(K+2)},\max\left\{2\cdot g_\varepsilon^{(K+1)},h^M\circ g_\varepsilon^{(K+1)}\right\}\right)
\end{equation*}
such that $\beta_{m,n}^l\leq\varepsilon/4$ holds whenever we have $m,n\in[N_0,N_0+g_\varepsilon^{(K+2)}(N_0)]$ as well as $l\leq 2\cdot g_\varepsilon^{(K+1)}(N_0)$ or $l\leq h^M\circ g_\varepsilon^{(K+1)}(N_0)$. The proof of Proposition~\ref{prop:KobMiy-Lem2} yields an $N\in[N_0,g_\varepsilon^{(K+1)}(N_0)]$ with $|\theta_m^f-\theta_n^f|<\varepsilon$ for all~$m,n\in[N,N+g(N)]$. As before we see $N\leq\Theta^B(\varepsilon,g,h)$. In view of $h(N)\leq h^M\circ g_\varepsilon^{(K+1)}(N_0)$ and
\begin{gather*}
N+g(N)\leq g_\varepsilon(N)\leq g_\varepsilon^{(K+2)}(N_0),
\end{gather*}
all the desired inequalities $\beta_{m,n}^k\leq\varepsilon/4$ are available as well.
\end{proof}

In their proof of~\cite[Lemma~3]{KobMiy}, Kobayasi and Miyadera show that $(S_nT^nx)$ is a Cauchy sequence. We now provide a rate of metastability.

\begin{definition}\label{def:metastab-Delta}
For $\varepsilon>0$ and $g:\mathbb N\to\mathbb N$ we set
\begin{equation*}
g_\varepsilon'(n):=\ceil*{\frac{6b\cdot(n+g(n))}{\delta(\varepsilon)}}\quad\text{with}\quad\delta(\varepsilon):=\min\left\{b,\frac{\varepsilon}{4},\frac{\varepsilon}{8}\cdot\eta\left(\frac{\varepsilon}{2b}\right)\right\}.
\end{equation*}
We then define
\begin{equation*}
\Delta(\varepsilon,g):=\Theta^B(\delta(\varepsilon),g+g_\varepsilon',\operatorname{Id}+g+2\cdot g_\varepsilon'),
\end{equation*}
with $\operatorname{Id}(n)=n$ and for $\delta(\varepsilon)$ as above.
\end{definition}

As promised, we have the following:

\begin{proposition}\label{prop:Delta}
For any $\varepsilon>0$ and $g:\mathbb N\to\mathbb N$ there is an~$N\leq\Delta(\varepsilon,g)$ such that we have $\norm{S_mT^mx-S_nT^nx}<\varepsilon$ for all $m,n\in[N,N+g(N)]$.
\end{proposition}

\begin{proof}
Pick a fixed point $f=Tf\in C\subseteq B_{b/2}$, as justified by the standing assumptions from the introduction. In the conclusion of the proposition, we can replace both expressions~$S_kT^kx$ by $y_k:=S_kT^kx-f$. Note that we get $\norm{y_k}=\theta_k^f$ in the notation from above. Corollary~\ref{cor:Theta^B} yields an~$N\leq\Delta(\varepsilon,g)$ with
\begin{equation}\def\arraystretch{1.5}\label{eq:sim-metastab-for-Delta}
\left.
\begin{array}{r}
|\theta_m^f-\theta_n^f|<\delta(\varepsilon)\\ \text{and }\beta_{m,n}^l\leq\frac{\delta(\varepsilon)}{4}
\end{array}\right\}
\quad\text{when}\quad
\left\{
\begin{array}{l}
m,n\in[N,N+g(N)+g'_\varepsilon(N)]\\ \text{and }l\leq N+g(N)+2\cdot g'_\varepsilon(N).
\end{array}\right.
\end{equation}
We must show $\norm{y_m-y_n}<\varepsilon$ for arbitrary~$m,n\in[N,N+g(N)]$, say with $m\leq n$. For $\delta(\varepsilon)$ as in Definition~\ref{def:metastab-Delta}, we put
\begin{equation*}
k:=\ceil*{\frac{6n\cdot b}{\delta(\varepsilon)}}\leq g_\varepsilon'(N)\quad\text{and}\quad K:=m+k\leq N+g(N)+g_\varepsilon'(N).
\end{equation*}
First assume that we have $\theta_K^f\leq\varepsilon/4$. Using~(\ref{eq:sim-metastab-for-Delta}), we then get
\begin{equation*}
\norm{y_m-y_n}\leq\norm{y_m}+\norm{y_n}=\theta^f_m+\theta^f_n<2\cdot(\theta^f_K+\delta(\varepsilon))\leq\varepsilon.
\end{equation*}
From now on we assume $\theta^f_K>\varepsilon/4$, which entails
\begin{equation}\label{eq:theta-times-to-plus}
(\theta^f_K+\delta(\varepsilon))\cdot\left(1-\eta\left(\frac{\varepsilon}{2b}\right)\right)<\theta_K^f-\delta(\varepsilon).
\end{equation}
Normalize $y_m$ and $y_n$ by dividing through $\theta^f_K+\delta(\varepsilon)$, and then apply the contrapositive of~(\ref{eq:modulus-unif-conv}) with $\varepsilon/(2b)$ at the place of~$\varepsilon$. Due to $\theta_K^f,\delta(\varepsilon)\leq b$ this yields
\begin{equation*}
\norm{\frac{y_m+y_n}{2}}>(\theta^f_K+\delta(\varepsilon))\cdot\left(1-\eta\left(\frac{\varepsilon}{2b}\right)\right)\,\,\Rightarrow\,\,\norm{y_m-y_n}<(\theta^f_K+\delta(\varepsilon))\cdot\frac{\varepsilon}{2b}\leq\varepsilon.
\end{equation*}
In view of~(\ref{eq:theta-times-to-plus}) it remains to show that we have $\norm{(y_m+y_n)/2}\geq\theta_K^f-\delta(\varepsilon)$. As in the proof of~\cite[Lemma~3]{KobMiy}, we first observe that $\norm{y_{l+1}-y_l}$ tends to zero: Since our standing assumption $C\subseteq B_{b/2}$ ensures $\norm{T^lx}\leq b/2$, we have
\begin{align*}
\norm{y_{l+1}-y_l}&=\norm{\left(\frac{1}{l}-\frac{1}{l(l+1)}\right)\cdot\sum_{i=1}^{l+1} T^{l+i}x-\frac{1}{l}\cdot\sum_{i=0}^{l-1} T^{l+i}x}\leq{}\\
{}&\leq\frac{1}{l}\cdot\norm{T^{2l+1}x+T^{2l}x-T^lx}+\frac{1}{l(l+1)}\cdot\sum_{i=1}^{l+1}\norm{T^{l+i}x}\leq\frac{2b}{l}.
\end{align*}
Summing up yields $\norm{y_{l+i}-y_l}\leq 2b\cdot i/l$. Applied to $l=m+k=K$ and $i=n-m$, we obtain
\begin{equation*}
\norm{y_{n+k}-y_K}\leq 2b\cdot\frac{n-m}{m+k}\leq 2b\cdot\frac{n}{k}\leq\frac{\delta(\varepsilon)}{3}\leq\delta(\varepsilon).
\end{equation*}
We can now use the reverse triangle inequality to infer
\begin{equation}\label{eq:reverse-triangle-delta}
\norm{y_{m+k}+y_{n+k}}\geq 2\cdot\norm{y_K}-\norm{y_{n+k}-y_K}\geq 2\cdot\theta_K^f-\delta(\varepsilon).
\end{equation}
On the other hand, the proof of~\cite[Lemma~3]{KobMiy} shows
\begin{equation*}
\norm{y_{m+k}+y_{n+k}}\leq\norm{y_m+y_n}+\frac{3n}{m+k}\cdot b+\frac{2}{m+k}\cdot\sum_{i=0}^{m+k-1}\beta_{m,n}^{k+i}.
\end{equation*}
For $i\leq m+k-1$ we have $k+i\leq m+2k\leq N+g(N)+2\cdot g'_\varepsilon(N)$. By~(\ref{eq:sim-metastab-for-Delta}) and the choice of~$k$ we can thus infer $\norm{y_{m+k}+y_{n+k}}\leq\norm{y_m+y_n}+\delta(\varepsilon)$. Together with inequality~(\ref{eq:reverse-triangle-delta}) we can conclude
\begin{equation*}
\norm{y_m+y_n}\geq\norm{y_{m+k}+y_{n+k}}-\delta(\varepsilon)\geq 2\cdot(\theta^f_K-\delta(\varepsilon)),
\end{equation*}
as needed to establish the open claim.
\end{proof}

We have just analysed the proof that $(S_nT^nx)$ is a Cauchy sequence. The limit~$y$ of this sequence is a fixed point of~$T$, as shown by Kobayasi and Miyadera (still in the proof of~\cite[Lemma~3]{KobMiy}). One might expect that a quantitative version of this result consists in (metastable) bounds on the norms~$\norm{TS_nT^nx-S_nT^nx}$. However, in the proof of~\cite[Theorem~1]{KobMiy}, Kobayasi and Miyadera use that we have $T^ly=y$ for arbitrary~$l\in\mathbb N$ (rather than just for~$l=1$). In order to reflect this fact, we will bound $\norm{T^lS_nT^nx-S_nT^nx}$ with a metastable dependency on~$l$.

\begin{definition}\label{def:Psi}
First extend Definition~\ref{def:metastab-Delta} by setting
\begin{equation*}
\Delta^B(\varepsilon,g,h):=\Theta^B(\delta(\varepsilon),g+g_\varepsilon',\max\{\operatorname{Id}+g+2\cdot g_\varepsilon',h\}).
\end{equation*}
For $\varepsilon>0$ and $g,h:\mathbb N\to\mathbb N$ we now put
\begin{equation*}
\Psi(\varepsilon,g,h):=\Delta^B\left(\frac{\varepsilon}{4},g''_{\varepsilon,h},h\right)\quad\text{with}\quad g''_{\varepsilon,h}(N):=\max\left\{g(N),\ceil*{\frac{4b\cdot h(N)}{\varepsilon}}\right\}.
\end{equation*}
\end{definition}

It is straightforward to see that $\Delta^B$ combines Proposition~\ref{prop:Delta} with Corollary~\ref{cor:Theta^B} (by the proof of the proposition and $\delta(\varepsilon)\leq\varepsilon/4$; cf.~the proof of the corollary):

\begin{corollary}\label{cor:Delta^B}
For any $\varepsilon>0$ and $g,h:\mathbb N\to\mathbb N$ there is an $N\leq\Delta^B(\varepsilon,g,h)$ with $\beta_{m,n}^l\leq\varepsilon/16$ and $\norm{S_mT^m-S_nT^nx}<\varepsilon$ for $m,n\in[N,N+g(N)]$ and $l\leq h(N)$.
\end{corollary}

We are ready to show that $\Psi$ validates statement~(\ref{eq:SnTn-FP-metastab}) from the introduction:

\begin{theorem}\label{thm:Psi}
For any $\varepsilon>0$ and $g,h:\mathbb N\to\mathbb N$ there is an $N\leq\Psi(\varepsilon,g,h)$ such that we have $\norm{T^lS_nT^nx-S_nT^nx}<\varepsilon$ for all $n\in[N,N+g(N)]$ and $l\leq h(N)$.
\end{theorem}
\begin{proof}
For $g_{\varepsilon,h}''$ as in Definition~\ref{def:Psi}, the corollary yields an $N\leq\Psi(\varepsilon,g,h)$ with
\begin{equation}\def\arraystretch{1.5}\label{eq:metastab-for-Psi}
\left.
\begin{array}{r}
\norm{S_mT^mx-S_nT^nx}<\frac{\varepsilon}{4}\\ \text{and }\beta_n^l<\frac{\varepsilon}{4}
\end{array}\right\}
\quad\text{when}\quad
\left\{
\begin{array}{l}
m,n\in[N,N+g_{\varepsilon,h}''(N)]\\ \text{and }l\leq h(N).
\end{array}\right.
\end{equation}
To establish the conclusion of the theorem, we consider arbitrary~$n\in[N,N+g(N)]$ and~$l\leq h(N)$. Set $m:=N+g_{\varepsilon,h}''(N)$ and observe
\begin{equation*}
\norm{T^lS_mT^mx-T^lS_nT^nx}\leq\norm{S_mT^mx-S_nT^nx}<\frac{\varepsilon}{4}.
\end{equation*}
Using the triangle inequality, we can conclude
\begin{equation*}
\norm{T^lS_nT^nx-S_nT^nx}\leq\norm{T^lS_mT^mx-S_mT^mx}+\frac{\varepsilon}{2}.
\end{equation*}
In other words, we have reduced the claim for~$n$ to the claim for the given~$m$ (at the cost of a smaller~$\varepsilon$). The point is that~$m$ is large, so that $S_mT^mx$ is a better approximation of the fixed point $y=Ty$ that is considered in Kobayasi and Miyadera's proof of~\cite[Lemma~3]{KobMiy}. Also by the triangle inequality, we get
\begin{equation*}
\norm{T^lS_mT^mx-S_mT^mx}\leq\norm{S_mT^{l+m}x-S_mT^mx}+\beta_m^l.
\end{equation*}
Due to cancellations between the {C}es\`aro sums, we have
\begin{equation*}
\norm{S_mT^{l+m}x-S_mT^mx}\leq\frac{1}{m}\cdot\left(\sum_{i=m}^{m+l-1}\norm{T^{m+i}x}+\sum_{i=0}^{l-1}\norm{T^{m+i}x}\right).
\end{equation*}
By the standing assumption that $T:C\to C$ has domain $C\subseteq B_{b/2}$, this entails
\begin{equation*}
\norm{S_mT^{l+m}x-S_mT^mx}\leq\frac{l\cdot b}{m}\leq\frac{h(N)\cdot b}{g_{\varepsilon,h}''(N)}\leq\frac{\varepsilon}{4}.
\end{equation*}
Given that (\ref{eq:metastab-for-Psi}) provides $\beta_m^l<\varepsilon/4$, we can combine the previous inequalities to obtain $\norm{T^lS_nT^nx-S_nT^nx}<\varepsilon$, as desired.
\end{proof}

In order to obtain the simultaneous bounds from Corollaries~\ref{cor:Theta^B} and~\ref{cor:Delta^B}, we have modified $\Theta$ and $\Delta$ into $\Theta^B$ and $\Delta^B$, respectively. In the case of~$\Psi$, we get a simultaneous bound without additional modifications, by the proof of Theorem~\ref{thm:Psi}:

\begin{corollary}\label{cor:Psi-simult}
For any $\varepsilon>0$ and $g,h:\mathbb N\to\mathbb N$ there is an $N\leq\Psi(\varepsilon,g,h)$ such that $\norm{T^lS_nT^nx-S_nT^nx}<\varepsilon$ and $\norm{S_mT^mx-S_nT^nx}<\varepsilon/4$ and $\beta_n^l<\varepsilon/4$ hold for all $m,n\in[N,N+g(N)]$ and $l\leq h(N)$.
\end{corollary}

Finally, we can specify the map~$\Phi$ that validates~(\ref{eq:SnTk-univ-metastab}) from the introduction:

\begin{definition}\label{def:Phi}
First, let the map $(\varepsilon,n)\mapsto N_\varepsilon(n)$ be given by
\begin{equation*}
N_\varepsilon(n):=\max\left\{n,\ceil*{\frac{6n\cdot b}{\varepsilon}}\right\}.
\end{equation*}
For $\varepsilon>0$ and $g,h:\mathbb N\to\mathbb N$ we now set
\begin{equation*}
\Phi(\varepsilon,g,h):=N_\varepsilon\left(\Psi\left(\frac{\varepsilon}{2},g',h'\right)\right)\quad\text{with}\quad
\left\{
\begin{aligned}
g'(n)&:=N_\varepsilon(n)+g\left(N_\varepsilon(n)\right),\\[1ex]
h'(n)&:=g(n)+h\left(N_\varepsilon(n)\right).
\end{aligned}\right.
\end{equation*}
\end{definition}

As promised, we get the following:

\begin{theorem}\label{thm:Phi}
For any $\varepsilon>0$ and $g,h:\mathbb N\to\mathbb N$ there is an $N\leq\Phi(\varepsilon,g,h)$ such that $\norm{S_mT^mx-S_nT^kx}<\varepsilon$ holds for all $m,n\in[N,N+g(N)]$ and $k\leq h(N)$.
\end{theorem}
\begin{proof}
Use Corollary~\ref{cor:Psi-simult} to find an~$M\leq\Psi\left(\frac{\varepsilon}{2},g',h'\right)$ with
\begin{equation}\def\arraystretch{1.5}\label{eq:metastab-for-Phi}
\left.
\begin{array}{r}
\norm{T^lS_nT^nx-S_nT^nx}<\frac{\varepsilon}{2}\\
\text{and }\norm{S_mT^mx-S_nT^nx}<\frac{\varepsilon}{8}\\
\text{and }\beta_n^l<\frac{\varepsilon}{8}
\end{array}\right\}
\quad\text{when}\quad
\left\{
\begin{array}{l}
m,n\in[M,M+g'(M)]\\
\text{and }l\leq h'(M),
\end{array}\right.
\end{equation}
for $g'$ and $h'$ as in Definition~\ref{def:Phi}. To establish the conclusion of the theorem for $N:=N_\varepsilon(M)\leq\Phi(\varepsilon,g,h)$, consider arbitrary $m,n\in[N,N+g(N)]$ and $k\leq h(N)$. In view of $n\geq N\geq M$, the proof of~\cite[Theorem~1]{KobMiy} yields
\begin{equation}\label{eq:Psi-intermed}
\norm{S_nT^kx-S_mT^mx}\leq\frac{3M}{2n}\cdot b+\frac{1}{n}\cdot\sum_{i=M}^{n-1}\norm{S_MT^{k+i}x-S_mT^mx}.
\end{equation}
For each $i\in[M,n-1]$ we can write $k+i=M+l$ with $l\leq h'(M)$. Using the triangle inequality and~(\ref{eq:metastab-for-Phi}), it follows that $\norm{S_MT^{k+i}x-S_mT^mx}$ is smaller than
\begin{equation*}
\beta_M^l+\norm{T^lS_MT^Mx-S_MT^Mx}+\norm{S_MT^Mx-S_mT^mx}\leq\frac34\cdot\varepsilon.
\end{equation*}
By~(\ref{eq:Psi-intermed}) and the definition of~$N=N_\varepsilon(M)$, we can conclude
\begin{equation*}
\norm{S_nT^kx-S_mT^mx}<\frac{3M}{2\cdot N_\varepsilon(M)}\cdot b+\frac34\cdot\varepsilon\leq\varepsilon,
\end{equation*}
just as the theorem claims.
\end{proof}

Finally, we record a simultaneous rate of metastability for statements~(\ref{eq:SnTk-univ-conv}-\ref{eq:S_nx-asymptotic-reg}) from the introduction. As explained in the paragraph after statement~(\ref{eq:SnTk-univ-metastab}),  the original result of Kobayasi and Miyadera~\cite{KobMiy} (stated as Theorem~\ref{thm:KobMiy} above) can be recovered as an immediate consequence of the following.

\begin{corollary}\label{cor:simultaneous-1to3}
For any $\varepsilon>0$ and $g,h:\mathbb N\to\mathbb N$ there is an $N\leq\Phi(2\varepsilon/5,g,h)$ such that we have
\begin{equation*}
\max\left\{\norm{S_mT^mx-S_nT^kx},\norm{T^lS_nT^nx-S_nT^nx},\norm{T^lS_nT^kx-S_nT^kx}\right\}<\varepsilon
\end{equation*}
for all $m,n\in[N,N+g(N)]$ and $k,l\leq h(N)$.
\end{corollary}
\begin{proof}
By the proof of Theorem~\ref{thm:Phi} we obtain an $N\leq\Phi(2\varepsilon/5,g,h)$ such that we have $\norm{S_mT^mx-S_nT^kx}<2\varepsilon/5$ and $\norm{T^lS_nT^nx-S_nT^nx}<\varepsilon/5$ for all numbers $m,n\in[N,N+g(N)]$ and $k,l\leq h(N)$. Given that~$T$ is nonexpansive, we get
\begin{equation*}
\norm{T^lS_nT^k-S_nT^k}\leq 2\cdot\norm{S_nT^kx-S_NT^Nx}+\norm{T^lS_NT^Nx-S_NT^Nx}<\varepsilon,
\end{equation*}
as desired.
\end{proof}

\section{Bounds on asymptotic isometry}\label{sect:asymp-isom-Hilbert}

Theorem~\ref{thm:KobMiy} (due to Kobayasi and Miyadera) involves the condition that $T$ is asymptotically isometric on~$\{x\}$, or explicitly: that the values $\alpha^i_n=\norm{T^nx-T^{n+i}x}$ converge for $n\to\infty$ uniformly in~$i\in\mathbb N$. On the quantitative side, this corresponds to one of our standing assumptions: throughout the previous sections, we have assumed that we are given a map $(\varepsilon,g,h)\mapsto A(\varepsilon,g,h)$ that validates statement~(\ref{eq:metastab-alpha}). In the introduction, we have mentioned three cases in which a suitable~$A$ can be constructed: The first case (where $T$ is asymptotically regular and satisfies Witt\-mann's condition) is covered in a previous paper by the second author~\cite{kohlenbach-odd-operators}. Two further constructions in different cases are provided in the present section. Note that each construction yields a version of Theorems~\ref{thm:Psi} and~\ref{thm:Phi} in which $\Phi$ and $\Psi$ do no longer depend on~$A$ (but possibly on new data such as $\Gamma$ in~(\ref{eq:modulus-total-bound}) below).

For the first part of this section, we consider a nonexpansive map $T:C\to C$ on a subset~$C$ of a Hilbert space. If~$T$ is odd on~$C=-C\ni 0$, then we clearly have
\begin{equation}\label{eq:Wittmann-condition}
T0=0\in C\quad\text{and}\quad\norm{Tx+Ty}\leq\norm{x+y}\text{ for all }x,y\in C.
\end{equation}
In the following we do not assume that~$T$ is odd but do require that it satisfies (\ref{eq:Wittmann-condition}). This condition has been studied by Wittmann~\cite{wittmann90} and plays an important role in several applications of proof mining (cf.~\cite{kohlenbach-odd-operators,kohlenbach-banach-geodesic-2016,safarik12}). Most standing assumptions from the introduction are not needed for the following, but we still assume $C\subseteq B_{b/2}$ (actually it suffices here to assume that 
$\| x\|\le \frac{b}{2}$).

Our quantitative analysis is based on the following result, which is implicit in the proof of~\cite[Theorem~2]{Brezis-Browder76} (cf.~also \cite[Theorem~2.3]{Bruck78}). Let us point out that the sequence of norms $\norm{T^nx}$ is non-increasing, since $T$ is nonexpansive with $T0=0$. We have already seen that the sequence $(\alpha^i_n)$ is non-increasing for each~$i\in\mathbb N$.

\begin{lemma}\label{lem:asymp-isom}
We have $(\alpha^i_m)^2-(\alpha^i_{m+k})^2\leq 4\cdot\left(\norm{T^mx}^2-\norm{T^{m+k+i}x}^2\right)$.
\end{lemma}
\begin{proof}
Using binomial expansion and the fact that $(\norm{T^nx})$ is non-increasing, 
we learn that $(\alpha^i_m)^2-(\alpha^i_{m+k})^2$ is bounded by
\begin{equation*}
2\cdot\left(\norm{T^mx}^2-\norm{T^{m+k+i}x}^2+\langle T^{m+k}x,T^{m+k+i}x\rangle-\langle T^mx,T^{m+i}x\rangle\right).
\end{equation*}
Hence the claim reduces to
\begin{equation*}
\langle T^{m+k}x,T^{m+k+i}x\rangle-\langle T^mx,T^{m+i}x\rangle\leq\norm{T^mx}^2-\norm{T^{m+k+i}x}^2.
\end{equation*}
To obtain the latter, iterate~(\ref{eq:Wittmann-condition}) to get $\norm{T^{m+k}x+T^{m+k+i}x}^2\leq\norm{T^mx+T^{m+i}x}^2$. Now consider binomial expansions, and use that $(\norm{T^nx})$ is non-increasing.
\end{proof}

In view of the standing assumption $C\subseteq B_{b/2}$, the values $\norm{T^nx}^2$ form a non-increasing sequence of reals in~$[0,b^2/4]$. Let us recall the known rate of metastability for such sequences. Given $g:\mathbb N\to\mathbb N$, we define $\widetilde g:\mathbb N\to\mathbb N$ by $\widetilde g(n):=n+g(n)$. The iterates $\widetilde g^i:\mathbb N\to\mathbb N$ are given by the recursive clauses
\begin{equation*}
\widetilde g^0(n):=n\quad\text{and}\quad\widetilde g^{i+1}(n):=\widetilde g(\widetilde g^i(n)).
\end{equation*}
By~\cite[Proposition~2.27]{kohlenbach-book} (cf.~also the beginning of Section~\ref{sect:fixed-points} above) we have
\begin{multline}\label{eq:metastab-non-increasing}
\left|\norm{T^mx}^2-\norm{T^nx}^2\right|<\varepsilon\quad\text{for some }N\leq\widetilde g^{\ceil*{b^2/(4\varepsilon)}}(0)\\ \text{ and all }m,n\in[N,N+g(N)],
\end{multline}
for any~$\varepsilon>0$ and $g:\mathbb N\to\mathbb N$. We will see that statement~(\ref{eq:metastab-alpha}) from the introduction holds with the following map $A_1$ at the place of~$A$.

\begin{definition}\label{def:A_1}
For $\varepsilon>0$ and $g,h:\mathbb N\to\mathbb N$ we set
\begin{equation*}
A_1(\varepsilon,g,h):=\widetilde{g+h}^{\ceil*{b^2/\varepsilon^2}}(0).
\end{equation*}
\end{definition}

As promised, we obtain the following:

\begin{proposition}
Consider a nonexpansive map $T:C\to C$ on a set $C\subseteq B_{b/2}$ in Hilbert space. If~$T$ satisfies~(\ref{eq:Wittmann-condition}), then any $\varepsilon>0$ and $g,h:\mathbb N\to\mathbb N$ admit a number $N\leq A_1(\varepsilon,g,h)$ with $|\alpha^i_m-\alpha^i_n|<\varepsilon$ for all $m,n\in[N,N+g(N)]$ and $i\leq h(N)$.
\end{proposition}
\begin{proof}
By (\ref{eq:metastab-non-increasing}) with $\varepsilon^2/4$ and $g+h$ at the place of $\varepsilon$ and $g$, respectively, we find a number $N\leq A_1(\varepsilon,g,h)$ with
\begin{equation*}
\left|\norm{T^mx}^2-\norm{T^lx}^2\right|<\frac{\varepsilon^2}{4}\quad\text{for all }m,l\in[N,N+g(N)+h(N)].
\end{equation*}
Given $m\leq n$ in $[N,N+g(N)]$ and $i\leq h(N)$, combine this inequality for $l=n+i$ with Lemma~\ref{lem:asymp-isom} for $k=n-m$. This yields
\begin{equation*}
(\alpha^i_m-\alpha^i_n)^2\leq(\alpha^i_m)^2-(\alpha^i_n)^2\leq 4\cdot\left(\norm{T^mx}^2-\norm{T^{n+i}x}^2\right)<\varepsilon^2
\end{equation*}
and hence $0\leq\alpha^i_m-\alpha^i_n<\varepsilon$, as desired.
\end{proof}

For the second part of this section, we return to the case of a Banach space. Our aim is to satisfy~(\ref{eq:metastab-alpha}) when $(T^nx)$ has a convergent subsequence. The latter entails that there are $M,N\in\mathbb N$ as in the following lemma. We point out that the lemma is a quantitative version of a step in Bruck's proof of~\cite[Theorem~2.4]{Bruck78}. Also note that the lemma establishes (\ref{eq:metastab-alpha}) whenever we have $N\leq A(\varepsilon,g,h)=A(\varepsilon,g)$, independently of the function $h:\mathbb N\to\mathbb N$ that provides the bound $i\leq h(N)$ in~(\ref{eq:metastab-alpha}).

\begin{lemma}\label{lem:isom-from-subseq}
For $\varepsilon>0$ and $g:\mathbb N\to\mathbb N$, assume that we are given $M,N\in\mathbb N$ with $M\geq N+g(N)$ and $\norm{T^Mx-T^Nx}<\varepsilon/2$. We then have $|\alpha_m^i-\alpha^i_n|<\varepsilon$ for all numbers $m,n\in[N,N+g(N)]$ and $i\in\mathbb N$.
\end{lemma}
\begin{proof}
From $\norm{T^Mx-T^Nx}<\varepsilon/2$ we get $\norm{T^{M+i}x-T^{N+i}x}<\varepsilon/2$, as $T$ is nonexpansive by a standing assumption. We recall $\alpha^i_N=\norm{T^Nx-T^{N+i}x}$ to conclude
\begin{equation*}
\alpha^i_N\leq\norm{T^Nx-T^Mx}+\norm{T^Mx-T^{M+i}x}+\norm{T^{M+i}x-T^{N+i}x}<\varepsilon+\alpha^i_M.
\end{equation*}
The claim follows since $\alpha^i_n$ is decreasing in~$n$ (again because $T$ is nonexpansive).
\end{proof}

The assumption that $(T^nx)$ has a convergent subsequence is satisfied when the domain~$C$ of our map $T:C\to C$ (or just the set $\{T^nx\,|\,n\in\mathbb N\}$) is compact or -- equivalently -- (closed and) totally bounded. On the quantitative side, this last property can be witnessed by a modulus of total boundedness (in the sense of P.~Gerhardy~\cite{gerhardy08}), i.\,e., by a function $\gamma:(0,\infty)\to\mathbb N$ that satisfies the following: for any $\varepsilon>0$ and any sequence $(x_i)$ in~$C$, we have $\norm{x_j-x_i}\leq\varepsilon$ for some $i<j\leq\gamma(\varepsilon)$ (see \cite{kohlenbach-leustean-nicolae_fejer} 
for a metatheorem on uniform bound extractions for spaces given 
with such a modulus). We make an (apparently) weaker assumption, where $\gamma$ may depend on the sequence and only the smaller index is controlled: for the following, we assume that we are given a map $(\varepsilon,g)\mapsto\Gamma(\varepsilon,g)$ that guarantees
\begin{equation}\label{eq:modulus-total-bound}
\norm{T^{g(j)}x-T^{g(i)}x}<\varepsilon\qquad\text{for some $i\leq\Gamma(\varepsilon,g)$ and some $j>i$}.
\end{equation}
Recall the notation $\widetilde g$ and $\widetilde g^{i}$ from the paragraph before Definition~\ref{def:A_1}.

\begin{definition}
For $\varepsilon>0$ and $g:\mathbb N\to\mathbb N$ we put
\begin{equation*}
A_2(\varepsilon,g):=\widetilde g^K(0)\quad\text{with}\quad K:=\Gamma\left(\frac{\varepsilon}{2},l\mapsto\widetilde g^l(0)\right).
\end{equation*}
Here $(\varepsilon,g)\mapsto\Gamma(\varepsilon,g)$ is a given map that satisfies (\ref{eq:modulus-total-bound}).
\end{definition}

To conclude, we show that (\ref{eq:metastab-alpha}) holds with $(\varepsilon,g,h)\mapsto A_2(\varepsilon,g,h):=A_2(\varepsilon,g)$ at the place of~$A$. Given that $A_2$ is independent of~$h:\mathbb N\to\mathbb N$, the condition $i\leq h(N)$ can be dropped in the present situation.

\begin{proposition}
Assume that $T:C\to C$ is a nonexpansive map on a totally bounded subset of a Banach space, where the total boundedness for the sequences 
$(T^{g(n)}x)$ is witnessed by a given map $(\varepsilon,g)\mapsto\Gamma(\varepsilon,g)$ as in~(\ref{eq:modulus-total-bound}). For any $\varepsilon>0$ and $g:\mathbb N\to\mathbb N$ we can then find an $N\leq A_2(\varepsilon,g)$ such that $|\alpha_m^i-\alpha_n^i|<\varepsilon$ holds for all $m,n\in[N,N+g(N)]$ and $i\in\mathbb N$.
\end{proposition}
\begin{proof}
Our modulus of total boundedness provides $i\leq\Gamma(\varepsilon/2,l\mapsto\widetilde g^l(0))$ and $j>i$ such that $M:=\widetilde g^j(0)$ and $N:=\widetilde g^i(0)$ satisfy $\norm{T^Mx-T^Nx}<\varepsilon/2$. Note that the function $l\mapsto\widetilde g^l(0)$ is increasing because of $\widetilde g(n)=n+g(n)\geq n$. This allows us, first, to infer $N\leq A_2(\varepsilon,g)$ from the definition of~$A_2$. Secondly, we learn
\begin{equation*}
M=\widetilde g^j(0)\geq\widetilde g^{i+1}(0)=\widetilde g(\widetilde g^i(0))=\widetilde g(N)=N+g(N),
\end{equation*}
so that we can conclude by Lemma~\ref{lem:isom-from-subseq}.
\end{proof}
\begin{remark} 
In the finite dimensional case $X:=\mathbb R^N$ endowed 
with the Euclidean norm, 
a modulus of total boundedness for $B_{b/2}$ can be taken as 
\[ \gamma(\varepsilon):=\left\lceil \frac{\sqrt{N}b}{\varepsilon}\right\rceil, 
\]
see \cite[Example 2.8]{kohlenbach-leustean-nicolae_fejer}, and so we may 
take $\Gamma(\varepsilon,g):=\gamma(\varepsilon/2)$ independently of $g,T,x$.
\end{remark}
\noindent
{\bf Final Comment:} An inspection of the proofs above shows that the quantitative results do not depend on the assumptions of $X$ being complete or $C$ 
being closed. \\[1mm]
\noindent
{\bf Acknowledgment:} The second author was supported by the German Science 
Foundation (DFG KO 1737/6-1 and DFG KO 1737/6-2).

\end{document}